\documentclass[11pt, a4paper]{article}
\usepackage{amsthm,amsfonts,amsmath,amssymb,bookmark,appendix}
\usepackage{amsxtra, mathrsfs}
\usepackage{colordvi}
\usepackage[usenames,dvipsnames]{color}
\usepackage{tikz}
\usetikzlibrary{arrows}

\theoremstyle{plain}
\newtheorem{theorem}{Theorem}[section]
\newtheorem{lemma}[theorem]{Lemma}
\newtheorem{corollary}[theorem]{Corollary}
\newtheorem{proposition}[theorem]{Proposition}
\theoremstyle{remark}
\newtheorem{rem}[theorem]{\bf Remark}

\numberwithin{equation}{section}

\DeclareMathOperator{\spec}{spec}

\DeclareMathOperator{\divg}{div}
\DeclareMathOperator{\rot}{curl}
\DeclareMathOperator{\inner}{in}

\DeclareMathOperator{\Osc}{osc}
\DeclareMathOperator{\diam}{diam}

\renewcommand{\phi}{\varphi}
\newcommand{\norm}[1]{\lVert #1 \rVert}
\newcommand{\nn}{\nonumber}
\newcommand{\eps}{\varepsilon}
\newcommand{\R}{\mathbb{R}}
\newcommand{\F}{\mathcal{F}}

\newcommand{\B}{\mathscr{B}}
\newcommand{\C}{\mathbb{C}}
\newcommand{\D}{\mathscr{D}}
\newcommand{\Z}{\mathbb{Z}}

\begin{document}

\title{Estimates for the Lowest Eigenvalue of Magnetic Laplacians}
\date{}

\author{T. Ekholm\footnote{tomase@kth.se, Royal Institute of Technology KTH, Sweden}, 
H. Kova\v r\'{\i}k\footnote{hynek.kovarik@unibs.it, Universit\`a degli studi di Brescia, Italy}, 
F. Portmann\footnote{Corresponding Author; fabian@math.ku.dk, University of Copenhagen, Denmark}}

\maketitle

\begin{abstract}
	We prove various estimates for the first eigenvalue of the magnetic Dirichlet
	Laplacian on a bounded domain in two dimensions. When the magnetic field is constant, we give 
	lower and upper bounds in terms of geometric quantities of the domain. 
	We furthermore prove a lower bound for the first magnetic Neumann eigenvalue in the
	case of constant field.
\end{abstract}

\section{Introduction}
Let $\Omega$ be a bounded open domain in $\R^2$ and 
$B \in L^\infty_{loc}(\R^2)$ a real-valued function, the magnetic field.
To $B$ we associate a vector potential $A \in L^{\infty}(\Omega)$ such that 
$B = \rot A =\partial_1A_2 - \partial_2 A_1$ in $\Omega$, see Section~\ref{sec:first_type_est} for an explicit construction of $A$. 
The magnetic Dirichlet Laplacian on $\Omega$,
\begin{align}\label{def:dirichlet_ham}
	H^D_{\Omega,B} := (-i\nabla + A)^2
\end{align}
is then defined through the Friedrichs extension of the quadratic form
\begin{align*}
	h_{\Omega,A}^{D}[u] := \int_{\Omega}\left|(-i\nabla + A)\, u(x)\right|^2\,dx,
\end{align*}
on $C_c^{\infty}(\Omega)$.
Altogether, there is a huge amount of literature dealing with spectral 
properties of the operator $H^D_{\Omega,B}$ on bounded as well as
unbounded domains in $\R^2$. 
We refer to the \cite{AHS, CFKS} for an introduction on Schr\"odinger operators with magnetic fields.
Various estimates for sums and Riesz means of eigenvalues of 
$H^D_{\Omega,B}$ on bounded domains were established in \cite{elv, kw, llr, ls}.
Hardy-type inequalities for $h_{\Omega,A}^{D}$ were studied in \cite{bls,lw,timo}.
For a version of the well-known Faber-Krahn 
inequality in the case of constant magnetic field we refer to \cite{er}.

The main object of interest in this note will be the quantity
\begin{equation*}
	\lambda_1(\Omega, B) := \inf\spec H_{\Omega,B}^D 
	= \inf_{u \in C_c^{\infty}(\Omega)} \frac{h_{\Omega,A}^D[u]}{\norm{u}^2}\geq 0.
\end{equation*}
Since $\Omega$ is bounded, our conditions on $B$
imply that the form domain of $H_{\Omega,B}^D$ is $H_0^1(\Omega)$
and $\lambda_1(\Omega, B)$ 
is indeed the lowest eigenvalue of $H^D_{\Omega,B}$.
There exist two well-known lower bounds for $\lambda_1(\Omega,B)$.
By a commutator estimate, see e.g. \cite{AHS}, one obtains 
\begin{align}\label{eq:comm_est}
	h_{\Omega,A}^{D}[u] \geq \pm \int_{\Omega}B(x)\, |u(x)|^2\,dx.
\end{align}
For a constant magnetic field $B(x) = B_0$, inequality \eqref{eq:comm_est} yields
\begin{align}\label{eq:comm_est_const}
	\lambda_1(\Omega, B_0) \geq \pm B_0.
\end{align}
The pointwise diamagnetic inequality (see for example \cite[Theorem~7.21]{LL})
\begin{align}\label{eq:diamag_ineq}
	|(-i\nabla + A)\, u(x)| \geq |\nabla|u(x)||, \quad \text{for a.e.}\,\,x\in\Omega,
\end{align}
on the other hand tells us that 
\begin{align*}
	\inf_{u\in H^1_0(\Omega)} \frac{h_{\Omega,A}^D[u]}{\|u\|^2} \ 
	&\geq \inf_{u\in H^1_0(\Omega)} \frac{ \int_{\Omega}|\nabla|u||^2}{\|u\|^2}  
	= \inf_{\begin{subarray}{1} v \in H_0^1(\Omega)\\ v \geq 0 \end{subarray}} \frac{ \int_{\Omega}|\nabla v|^2}{\|v\|^2}\\
	& \geq \inf_{v\in H^1_0(\Omega)} \frac{ \int_{\Omega}|\nabla v|^2}{\|v\|^2}.
\end{align*}
This implies that
\begin{equation}\label{eq:diamag_est}
	\lambda_1(\Omega, B) \geq \lambda_1(\Omega,0).
\end{equation}
Under very weak regularity conditions on $B$
it was shown in \cite{he} that inequality \eqref{eq:diamag_est} is in fact strict; 
$\lambda_1(\Omega, B) > \lambda_1(\Omega,0)$. 

Let us briefly discuss the Neumann case. The quadratic form corresponding to the magnetic Neumann Laplacian
$H_{\Omega,B}^N$ is given by 
$$
	h_{\Omega,A}^N[u]:=\int_{\Omega}|(-i\nabla+A)u(x)|^2\,dx, 
$$
and the form domain is now $H^1(\Omega)$.
Again,
$$
	\mu_1(\Omega,B):= \inf\spec H_{\Omega,B}^N
$$ 
is the first eigenvalue of $H_{\Omega,B}^N$, provided $\Omega$ is sufficiently regular.
The estimate \eqref{eq:diamag_est} remains valid in the 
Neumann case (since \eqref{eq:diamag_ineq} holds a.e.), and gives:
\begin{align}\label{eq:diamag_est_n}
	\mu_1(\Omega,B) \geq \mu_1(\Omega,0) = 0.	
\end{align}
The corresponding estimate \eqref{eq:comm_est} (resp. \eqref{eq:comm_est_const}) is 
a priori not available due to the different boundary conditions.
A lot of attention has been paid to the asymptotic behavior of $\mu_1(\Omega,B)$ for 
large values of the magnetic field, see e.g.~\cite{bo,fh1,lp,ra, si}.

\subsection{Overview of the Main Results}
A natural question which arises in this context is whether estimates \eqref{eq:comm_est}, 
\eqref{eq:comm_est_const} and \eqref{eq:diamag_est}
can be improved by adding a positive term to their righthand sides. 

It is clear that their combinations cannot be achieved by simple addition; already for the constant magnetic
field any lower bound of the type 
\begin{equation}\label{eq:fail_est}
	\lambda_1(\Omega, B_0) \geq \lambda_1(\Omega,0) + cB_0,
\end{equation}
with $c>0$ independent of $B_0$, must fail. Indeed, since the eigenfunction of 
$H_{\Omega,0}^D = -\Delta_{\Omega}^D$ relative to the eigenvalue $\lambda_1(\Omega,0)$ 
may be chosen real-valued, analytic perturbation theory yields 
\begin{equation} \label{eq:small_B_asymp}
	\lambda_1(\Omega, B_0) = \lambda_1(\Omega, 0) + \mathcal{O}(B_0^2), \quad B_0\to 0.
\end{equation}
This clearly contradicts \eqref{eq:fail_est} for $B_0$ small enough.

The main results of our paper are the following. In Section~\ref{sec:first_type_est} we give quantitative lower 
bounds on the quadratic form
\begin{align*}
	h_{\Omega,A}^D[u] \mp \int_{\Omega}B(x)\, |u(x)|^2\,dx,
\end{align*}
denoted by \textit{estimates of the first type}.
Estimates for the difference
$$
	\lambda_1(\Omega, B)-\lambda_1(\Omega, 0),
$$
referred to as \textit{estimates of the second type}, are studied in Section~\ref{sec:second_type_est}.
In both cases, particular attention will be devoted to the case of constant magnetic field.
Last but not least, we will also establish a lower bound of the second type for the lowest 
eigenvalue of the magnetic Neumann Laplacian in the case of constant magnetic field 
in Section~\ref{ssec:second_type_est_n}.

\subsubsection*{Notation:}
Given $x\in\Omega$ and $r>0$, we denote by $\B(x,r)$ the open disc of radius $r$ centered in $x$. 
We also introduce the distance function 
$$
	\delta(x) := {\rm dist}\, (x, \partial\Omega),
$$
and the in-radius of $\Omega$,
$$
	R_{\inner}  := \sup_{x\in\Omega} \delta(x).
$$
Finally, given a positive real number $x$ we denote by $[x]$ its integer part.  

\section{Estimates of the First Type}\label{sec:first_type_est}
In this section we will derive lower bounds on the 
forms
$$
	h_{\Omega,A}^D[u] \mp \int_{\Omega}B(x)\, |u(x)|^2\,dx.
$$
Instead of introducing a vector potential $A$ associated to $B$, we decide link both quantities through 
a so called \textit{super potential}, an approach that is well-known in the study of the Pauli operator, see e.g. 
\cite{ev}. For our magnetic fields however, this approach is equivalent  with
the standard definition given in the introduction.

Let $r>0$ be such that $\overline{\Omega} \subset \B(0, r)$. For any $B\in L^\infty_{loc}(\R^2)$ let
$$
	\F(B) := \left\{ \Psi :\R^2 \to \R \,:\,   \Delta\Psi = B \ \textrm{ in }\  \B(0, r) \right\}
$$
be the family of super potentials associated to $B$. Note that $\F(B)$ is not empty. Indeed, the function
$$
	\Psi_0(x) = \frac{1}{2\pi} \int_{\B(0,r)} \log|x-y|\,B(y)\,dy, \quad x\in \R^2,
$$
which is well defined in view of the regularity of $B$,
solves 
\begin{equation} \label{delta-psi} 
\Delta \Psi_0 (x) = 
\left\{
		\begin{array}{l@{\quad}l}
		B(x) &\quad  x\in\B(0,r) ,\\
		& \\
		0 &\quad  \text{elsewhere},
		\end{array}
		\right. 
\end{equation}
in the distributional sense. Since $B\in L^\infty_{loc}(\R^2)$, standard regularity theory implies that 
$\Psi_0 \in W^{2,p}(\B(0,r))$ for every $1\leq p<\infty$, see \cite[Thm.~9.9]{gt}. 
Moreover, for any $\Psi\in \F(B)$ the difference $\Psi-\Psi_0$ is a harmonic function in $\B(0,r)$. 
Hence for any $\Psi\in \F(B)$ and any $p\in[1,\infty)$ we have $\Psi \in W^{2,p}(\B(0,r))$. 
By Sobolev's embedding theorem it then follows that $\Psi$ is continuous on $\B(0,r)$, so
we may define the oscillation of $\Psi$ over $\Omega$;
\begin{align*}
	\Osc(\Omega, \Psi) = \sup_{x \in \Omega} \Psi(x) - \inf_{x \in \Omega} \Psi(x).
\end{align*}
Accordingly we set
$$
	\D(\Omega, B) : = \inf_{\Psi \in\F(B)} \Osc(\Omega, \Psi). 
$$
Note also that a vector field $A:\B(0,r)\to \R^2$ defined by 
$$
     A := (-\partial_2\Psi,\partial_1\Psi), \quad \Psi\in\F(B),
$$
belongs to $W^{1,p}(\B(0,r))$ for every $1 \leq p<\infty$, in view of the regularity of $\Psi$, and satisfies
$$
	\rot A(x) = B(x) \quad \textrm{ in } \B(0,r)
$$ 
in the distributional sense. Hence by the Sobolev embedding theorem we have $A\in L^\infty(\B(0,r))$
and furthermore $\divg A = 0$ almost everywhere on $\B(0,r)$.

\begin{theorem} \label{thm:general_B_est}
	Let $\Omega \subset \R^2$ be a bounded open domain and suppose that $B\in L_{loc}^{\infty}(\R^2)$.
	Then
	\begin{align}\label{eq:thm_general_B_est}
		h_{\Omega,A}^D[u] \geq \pm \int_{\Omega}B(x)\, |u(x)|^2\,dx  
		+ e^{-2\D(\Omega, B)}\, \lambda_1(\Omega,0) \int_\Omega |u(x)|^2\, dx
	\end{align}
	holds true for all $u\in C^\infty_c(\Omega)$.
\end{theorem}
\begin{proof}
	We first prove inequality \eqref{eq:thm_general_B_est} with the plus sign on the right hand side. 
	To do so, we pick any $\Psi \in \F(B)$ and perform the ground state substitution $u(x) =: v(x)\, e^{-\Psi(x)}$
	and obtain, after a relatively lengthy (but straightforward) calculation,
	\begin{align} \label{eq:gs_rep}
		h_{\Omega,A}^D[u] - \int_{\Omega}B(x) \, |u(x)|^2\,dx 
		= \int_{\Omega}e^{-2\Psi}|(-i\partial_1 - \partial_2)v|^2\,dx.
	\end{align}
	Next, we have
	\begin{align*}
		\int_{\Omega}e^{-2\Psi}|(-i\partial_1 - \partial_2)v|^2\,dx &\geq e^{-2\sup_{x \in \Omega}\Psi(x)}
		\int_{\Omega}|(-i\partial_1 - \partial_2)v|^2\,dx\\
		&= e^{-2\sup_{x \in \Omega}\Psi(x)}\int_{\Omega}|\nabla v|^2\,dx,
	\end{align*}
	where in the last step we used that
	\begin{align*}
		\int_{\Omega}\left((\partial_1 v)\partial_2 \bar{v} - (\partial_1 \bar{v})\partial_2 v\right) \,dx = 0.
	\end{align*}
	It then follows that
	\begin{align*}
		h_{\Omega,A}^D[u] - \int_{\Omega}B(x)\, |u(x)|^2\,dx &\geq e^{-2\sup_{x \in \Omega}\Psi(x)}\int_{\Omega}|\nabla v|^2\,dx\\
		&\geq e^{-2\sup_{x \in \Omega}\Psi(x)}\,\lambda_1(\Omega,0)\int_{\Omega}|v|^2\,dx\\
		&\geq e^{-2\Osc(\Omega, \Psi)}\, \lambda_1(\Omega,0)\int_{\Omega}|u|^2\,dx.
	\end{align*}
	To prove the corresponding lower bound with the minus sign in front of $B$ on the right hand side, we note that the substitution
	$u(x) =: w(x)\, e^{\Psi(x)}$ gives
	\begin{align*}
		h_{\Omega,A}^D[u] + \int_{\Omega}B(x)\, |u(x)|^2\,dx 
		= \int_{\Omega} e^{2\Psi}\, |(-i\partial_1 + \partial_2)w|^2\,dx.
	\end{align*}
	Moreover, since $\Osc(\Omega, -\Psi)=\Osc(\Omega, \Psi)$, the same procedure as above 
	gives an identical lower bound.
	To complete the proof of \eqref{eq:thm_general_B_est} it now suffices to optimize the right hand side with respect to 
	$\Psi \in\F(B)$, keeping in mind that the spectral properties of $h_{\Omega,A}^D$ only depend on $B$.
\end{proof}

\subsection{Estimates for the Constant Magnetic Field}
In this section we consider the case of a constant magnetic field:  
$$
	B(x) = B_0 > 0.
$$
Clearly, all $\Psi \in \F(B_0)$ are smooth, and optimizing estimate \eqref{eq:thm_general_B_est}
amounts to minimizing the oscillation of $B_0\widetilde{\Psi}$, where $\widetilde{\Psi}$ satisfies
\begin{align*}
	\Delta \widetilde{\Psi} = 1 \quad \textrm{ in } \Omega.
\end{align*}
The optimal $\widetilde{\Psi}$ depends very much on the geometry of $\Omega$, 
and we start with a rather general result. 

We pick any point $x_0\in\Omega$ and a rotation $R(x_0,\theta) \in SO(2)$, parametrized by an angle $\theta \in [0,2\pi)$ and 
center of rotation $x_0$.
Set
\begin{align}\label{def:ell_theta}
	\ell(\Omega, x_0, \theta) := \sup_{x \in R(x_0,\theta)\Omega} x_2  - \inf_{x \in R(x_0, \theta)\Omega}x_2,
\end{align}
the maximal $x_2$-distance of the rotated set $R(x_0,\theta)\Omega$. The quantity $\ell(\Omega)$ is then defined as follows:
\begin{align}\label{def:ell_omega}
	\ell(\Omega) := \inf_{\theta \in [0,2\pi)}\ell(\Omega, x_0, \theta).
\end{align}
It is easily seen that $\ell(\Omega)$ is independent of the choice of $x_0\in\Omega$ and finite, since $\Omega$ is bounded.

\begin{theorem}\label{thm:const_B}
	Let $\Omega \subset \R^2$ a bounded open domain. Then
	\begin{align}\label{eq:thm_const_B_est}
		\lambda_1(\Omega,B_0) \geq B_0 + e^{-\frac{B_0}{4}\, \ell(\Omega)^2}\lambda_1(\Omega,0).
	\end{align}
\end{theorem}
\begin{proof}
	In view of estimate \eqref{eq:thm_general_B_est}
	we have
	\begin{align}\label{eq:any_A}
		\lambda_1(\Omega,B_0) \geq B_0 + e^{-2\Osc(\Omega,\Psi)}\, \lambda_1(\Omega,0), 
		\quad \forall\ \Psi\in\F(B_0). 
	\end{align}
	The rotational symmetry of the problem allows us to assume that $\Omega$ has 
	been rotated such that 
	$$
		\ell(\Omega) = \sup_{x \in \Omega}x_2 - \inf_{x \in \Omega}x_2. 
	$$
	Let $\alpha:=  \inf_{x \in \Omega}x_2$ and $\beta:=  \sup_{x \in \Omega}x_2$. 
	We then chose the super potential
	\begin{align}\label{eq:def_Phi_B_const}
		\Psi(x_1,x_2) = \frac{B_0}{2}(x_2 - a)^2,
	\end{align}
	where $a$ is a free parameter. 
	Observe that we may assume that the entire line between $\alpha$ and $\beta$ is contained
	in $\Omega$, because the oscillation of a function over a domain can only increase as the domain is increased
	(and $\Psi$ is globally well-defined).
	Next, we calculate $\Osc(\Omega,\Psi)$ and minimize the result with respect to $a$. 
	A direct calculation shows that the best choice is $a = (\alpha+\beta)/2$, which gives 
	$$
		\Osc(\Omega,\Psi) \leq \frac{B_0}{8}\, (\beta-\alpha)^2 =  \frac{B_0}{8}\,\ell(\Omega)^2.
	$$
	This in combination with \eqref{eq:any_A} implies \eqref{eq:thm_const_B_est}.
\end{proof}

\begin{rem}
	It was shown in \cite{er,er2} that for any $\eps>0$ there exists a constant $C(\eps)$ such that
	\begin{align} \label{lowerb-erd}
		\lambda_1(\B(0,R),B_0) \geq B_0 + \frac{C(\eps)}{R^2}e^{-B_0(\frac 12+\eps)R^2}, 
	\end{align}
	Together with the Faber-Krahn inequality \cite{er} this yields
	\begin{align}\label{eq:faber_krahn_est}
		\lambda_1(\Omega,B_0) 
		&\geq \lambda_1(\B(0,R),B_0)\nn\\
		&\geq B_0 + \frac{C(\eps)}{R^2}e^{-B_0(\frac 12+\eps)R^2}, \quad \eps>0,
	\end{align}
	where $R$ is such that $|\B(0,R)| = |\Omega|$.
	It is clear that \eqref{eq:thm_const_B_est} is an improvement of the estimate
	\eqref{eq:faber_krahn_est} for domains
	that are geometrically very far from the disc, as for example very wide rectangles or thin ellipses.
\end{rem}

\begin{proposition}\label{prop:geom_est}
	Let $\Omega \subset \R^2$ be any bounded convex open domain, then
	\begin{equation}\label{eq:inr_bound}
		2\, R_{\inner}  \leq \ell(\Omega)  \leq 3\, R_{\inner} .
	\end{equation}
\end{proposition}
\begin{proof}
	Let $\B \subset\Omega$ be a disc of radius $R_{\inner}$ contained in $\Omega$. Independently of $\Omega$ being convex or 
	not we have
	\begin{equation}\label{eq:aux_1}
		\ell (\Omega, x_0, \theta) \geq 2 R_{\inner},  \qquad \forall\ \theta\in [0,2\pi),\,x_0\in\Omega.
	\end{equation}
	This follows directly from \eqref{def:ell_theta}; for any $\theta\in [0,2\pi)$ and $x_0\in\Omega$ 
	we have that  $\ell (\Omega, x_0, \theta)$ is larger or equal to the length of the intersection of 
	$\Omega$ with the vertical line passing through the center of $\B$. 
	The latter is obviously  larger or equal to $ 2 R_{\inner}$, hence equation \eqref{eq:aux_1}.
	
	It remains to prove the second inequality of \eqref{eq:inr_bound}. 
	Let $\B \subset\Omega$ be a disc of radius $R_{\inner}$. 
	Assume first that $\partial \B \cap \partial \Omega$ contains at least two distinct points 
	$P_1$ and $P_2$ and that the vectors $\overline{OP_1}$ and $\overline{OP_2}$ are linearly 
	dependent. By convexity, $\Omega$ is contained in an infinite rectangle of height $2 R_{\inner}$, 
	and $\ell (\Omega) \leq 3 R_{\inner}$.
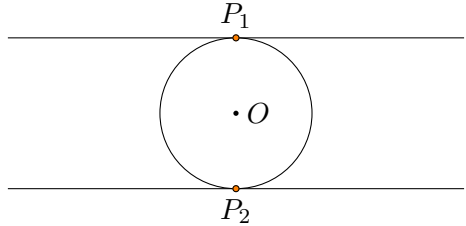
	\begin{center}
	\begin{tikzpicture}[line/.style={thick}]
	    \draw[black, thin] (0,0) circle(1);
	    \draw[fill, thick] (0,0) circle(0.02) node[right] {$O$};
	    \draw[black, thin] (-3,1) -- (3,1);
	    \draw[black, thin] (-3,-1) -- (3,-1);
	    \draw[orange, fill, thick] (0,1) circle(0.03);
	    \draw[black, thin] (0,1) circle(0.04) node[above] {$P_1$};
	    \draw[orange, fill, thick] (0,-1) circle(0.03);
	    \draw[black, thin] (0,-1) circle(0.04) node[below] {$P_2$};
	\end{tikzpicture}
	\end{center}
	Assume now that $\partial \B \cap \partial \Omega$ is distributed in such a way that there is a closed, 
	connected set $\Gamma \subset \partial B$ of length $\pi R_{\inner}$ with the property that the distance
	\begin{align*}
		\rho(x) := \inf_{y \in \partial \Omega} |x - y|, \quad x \in \Gamma,
	\end{align*}
	is positive.
	
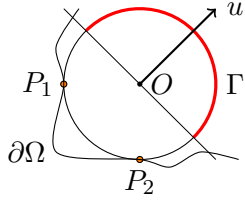
\begin{center}
	\begin{tikzpicture}[line/.style={thick}]
	    \draw[black, thin] (0,0) circle(1);
	    \draw[fill, thick] (0,0) circle(0.02) node[right] {$O$};
	    \draw[orange, fill, thick] (-1,0) circle(0.03);
	    \draw[black, thin] (-1,0) circle(0.04) node[left] {$P_1$};
	    \draw[orange, fill, thick] (0,-1) circle(0.03);
	    \draw[black, thin] (0,-1) circle(0.04) node[below] {$P_2$};
	    \draw[red, very thick] (0.707,-0.707) arc (-45:135:1);
	    \draw[black, very thin] (1,-1) -- (-1, 1);
	    \draw[thick,->] (0,0) -- (1,1) node[right] {$u$};
	    \draw (1,0) node[right] {$\Gamma$};
	    \draw plot [smooth] coordinates {(-0.8,1) (-1.1, 0.5) (-1, 0) (-1.1, -0.9) (0, -1) (0.4, -1.1) (0.8, -0.9) (1.3, -1)};
	    \draw (-1.1,-0.9) node[left] {$\partial \Omega$};        
	\end{tikzpicture}
	\end{center}
	Since $\rho$ is continuous and $\Gamma$ is closed there is an $\varepsilon$ such 
	that $\rho (x) \geq \varepsilon > 0$, for all $x \in \Gamma$. Hence we can move the disc $\B$ 
	a distance $\varepsilon/2$ much in the direction of $u$, such that $\overline{\B}$ becomes a 
	proper subset of $\Omega$. This contradicts that the inner radius of $\Omega$ is $R_{\inner}$.
	
	Assume that $\partial \B \cap \partial \Omega$ contains at least three points $P_1$, $P_2$ and $P_3$. 
	They must be distributed in such a way that there is no such $\Gamma$ as above. 
	Since $\Omega$ is convex, it is contained in a triangle given by the tangent lines of the intersection points.
	
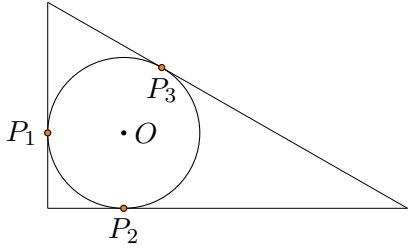
\begin{center}
	\begin{tikzpicture}[line/.style={thick}]
	    \draw[black, thin] (0,0) circle(1);
	    \draw[fill, thick] (0,0) circle(0.02) node[right] {$O$};
	    \draw[black, very thin] (-1,-1) -- (-1, 1.732);
	    \draw[black, very thin] (-1,-1) -- (3.732, -1);
	    \draw[black, very thin] (-1,1.732) -- (3.732, -1);
	    \draw[orange, fill, thick] (-1,0) circle(0.03);
	    \draw[black, thin] (-1,0) circle(0.04) node[left] {$P_1$};
	    \draw[orange, fill, thick] (0,-1) circle(0.03);
	    \draw[black, thin] (0,-1) circle(0.04) node[below] {$P_2$};
	    \draw[orange, fill, thick] (0.5,0.866) circle(0.03);
	    \draw[black, thin] (0.5,0.866) circle(0.04) node[below] {$P_3$};
	\end{tikzpicture}
	\end{center}
	For a triangle, $\ell (T)$ is given by the smallest height, which is maximized for the equilateral triangle. 
	Hence $\ell (\Omega) \leq 3 R_{\inner}$.
\end{proof}

Note that the second term on the righthand side of \eqref{eq:thm_const_B_est} decays exponentially fast to zero as 
$B_0$ tends to infinity.
This was in fact already observed in \cite[Remark~1.4.3]{fh2}, where the authors observed that 
\begin{align} \label{fh-general}
	\frac{\lambda_1(\Omega,B_0)}{B_0} = 1+\mathcal{O}(\exp(-\alpha B_0)), \quad B_0 \to \infty,
\end{align}
and $\alpha$ is a positive constant.
The optimal value of $\alpha$ is in general unknown, however it was conjectured in \cite[Remark~1.4.3]{fh2} that $\alpha$ is proportional to $R_{\inner}^2$. This is in agreement with Proposition~\ref{prop:geom_est} and the following result.

\begin{proposition}\label{prop:upper_bound}
	Let $\Omega \subset \R^2$ be a bounded open domain
	and suppose that $B_0  R_{\inner}^{2}\geq 4\, $. Then
	\begin{equation} \label{eq:upper_bound}
		\lambda_1(\Omega, B_0) \leq B_0 +  e\, B_0^2 \,  R_{\inner}^2 \ e^{-\frac{B_0}{2} R_{\inner}^2}\, . 
	\end{equation}
\end{proposition}
\begin{proof} 
	In view of \eqref{eq:gs_rep}, any $v \in H_0^1(\Omega)$ satisfies
	\begin{equation}\label{eq:test_func}
		\lambda_1(\Omega, B_0) 
			\leq B_0 + \frac{ \int_{\Omega}e^{-2\Psi}|(-i\partial_1 - \partial_2)v|^2\,dx}{ \int_{\Omega} e^{-2\Psi}|v|^2\,dx}.
		\end{equation}
	Without loss of generality we may assume that the largest disc contained in 
	$\Omega$ is centered in the origin. Hence $\B(0,  R_{\inner}) \subset\Omega$. 
	We choose the super potential in the form 
	$$
			\Psi(x) = \frac{B_0}{4} |x|^2,
	$$
	and apply inequality \eqref{eq:test_func} with 
	$$
		v(x) = \left\{
		\begin{array}{l@{\quad}l}
		1                   & \quad |x|  \leq R_{\inner}  -\frac{1}{R_{\inner} B_0},\\
		& \\
		R_{\inner} B_0 (R_{\inner}  - |x| )  &\quad R_{\inner}  -\frac{1}{R_{\inner} B_0}  < |x|  < R_{\inner},\\
		& \\
		0            & \quad  \text{elsewhere}. 
		\end{array}
		\right. 
	$$
	Obviously $v \in H^1_0(\Omega)$. Note also that since $v$ is real valued, we have 
	\mbox{$|(-i\partial_1 - \partial_2)v|^2= |\nabla v|^2$}. Performing both integrations in 
	\eqref{eq:test_func} in polar coordinates and taking into account the condition $B_0R_{\inner}^{2} \geq 4$,
	we find that 
	\begin{align*}
		\int_{\Omega}e^{-2\Psi}|(-i\partial_1 - \partial_2)v|^2\,dx 
		& =2\pi B_0^2 R_{\inner}^{2} \int_{R_{\inner}-\frac{1}{R_{\inner} B_0}}^{R_{\inner} } \ e^{-\frac{B_0}{2}\, r^2}\, r\, dr\\
		& = 2\pi B_0 R_{\inner}^{2}\  e^{-\frac{B_0}{2}\, R_{\inner}^2}  \left(\exp\left(1-\frac{1}{2 B_0^2 R_{\inner}^{2}}\right)-1\right) \\
		& \leq 2\pi (e-1)\, B_0 R_{\inner}^{2}\  e^{-\frac{B_0}{2} R_{\inner}^2} 
	\end{align*}
		and
	\begin{align*}
		\int_{\Omega}e^{-2\Psi}|v|^2\,dx 
		\geq  2\pi  \int_0^{R_{\inner}-\frac{1}{R_{\inner} B_0}}  e^{-\frac{B_0}{2}\, r^2}\, r\, dr 
		\geq \frac{2\pi}{B_0}\, (1-e^{-1})
	\end{align*}
	which proves \eqref{eq:upper_bound}.
\end{proof}

\noindent In the special case when $\Omega$ is a disc we have

\begin{proposition}\label{prop:disc_est}
	Let $\Omega=\B(0,R)$ and let $B_0 R^2 \geq 4$. Then
	\begin{equation} \label{disc}
		1+ \frac{j_{0,1}^2}{B_0 R^2} \ e^{-\frac 12 B_0 R^2} 
		\leq \frac{\lambda_1(\B(0,R),B_0)}{B_0}
		\leq 1+ e\, B_0  R^2 \ e^{-\frac 12 B_0 R^2} ,
	\end{equation}
	where $j_{0,1} \simeq 2.405$ is the first zero of the Bessel function $J_0$. 
\end{proposition}

\begin{proof}
	The upper bound is given by Proposition~\ref{prop:upper_bound}. 
	For the lower bound we use Theorem \ref{thm:general_B_est} with the super potential $\Psi(x) = \frac{B_0}{4} |x|^2$,
	which is in this case better suited to the geometry of the domain than the one used in Proposition~\ref{prop:geom_est}.
	Hence
	$$
		\Osc(\B(0,R), \Psi) = \frac{B_0 R^2}{4},
	$$ 
	and since $\lambda(\B(0,R), 0) = \frac{j_{0,1}^2}{R^2}$, 
	the lower bound in \eqref{disc} follows from \eqref{eq:thm_general_B_est}. 
\end{proof}

\begin{rem}
	In \cite[Prop.~4.4]{hm} it was stated that
	\begin{align} \label{eq-HM}
		\frac{\lambda_1(\B(0,R),B_0)}{B_0} \,\sim \, 1+ \frac{2^{3/2}}{\sqrt{\pi}}\ \sqrt{B_0}\ R\,e^{-\frac 12 B_0 R^2},
		\quad B_0 \to \infty.
	\end{align}
	B.~Helffer however pointed out to us that the argument used to establish the above was slightly flawed -- the above
	is only true if taken as a lower bound. 
	However, from Proposition~\ref{prop:disc_est} we easily see that
	\begin{align*}
		\lim_{B_0 \to \infty}\frac{\log\left [\lambda_1(\B(0,R),B_0)-B_0\right ]}{B_0R^2} = -\frac 12\, ,
	\end{align*}
	which confirms, up to a pre-factor, the asymptotic \eqref{eq-HM} stated in \cite[Prop.~4.4]{hm}. 
	Note also that the lower bound in \eqref{disc} improves qualitatively the lower 
	bound \eqref{lowerb-erd}, since it allows us to pass to the limit $\eps\to 0$.
\end{rem}

\section{Estimates of the Second Type}\label{sec:second_type_est}
\subsection{Dirichlet boundary conditions}\label{ssec:second_type_est_d}
In this section we are going to establish a lower bound on the difference $\lambda_1(\Omega,B) - \lambda_1(\Omega,0)$. Given a point $x\in\Omega$, we introduce the function
\begin{equation}\label{def:flux_local}
	\Phi(r,x) := \frac{1}{2\pi} \int_{\B(x,r)} B(x)\, dx, \quad 0\leq r \leq \delta(x),
\end{equation}
the flux through $\B(x,r)$.
The next result shows that as soon as the magnetic field is not identically zero in $\Omega$, 
the difference $\lambda_1(\Omega,B) - \lambda_1(\Omega,0)$ is strictly positive.

\begin{theorem}\label{thm:ev_diff}
	Let $\Omega$ be a bounded open domain in $\R^2$ and $B\in L_{loc}^\infty(\R^2)$. If $B$ is not 
	identically zero in $\Omega$, then there exists $y\in\Omega$ such that
	\begin{align}\label{eq:ev_diff}
		\lambda_1(\Omega,B) - \lambda_1(\Omega,0) \geq D(y,B),
	\end{align}
	where $D(y,B)>0$ is given by \eqref{def:D_y}.
\end{theorem}

For the proof of Theorem~\ref{thm:ev_diff} we are going to need the following elementary result. 
\begin{lemma}\label{lem:c_nbr_est}
	Let $z_1,z_2 \in\C$ and let $\beta, \gamma$ be positive constants. Then
	\begin{align*}
		\beta |z_1|^2 + \gamma |z_1 + z_2|^2 \geq \frac{\beta\gamma}{\beta+\gamma}\ |z_2|^2.
	\end{align*}
\end{lemma}

\begin{proof}
	Since 
	$$
		|\bar{z}_1z_2| + |z_1 \bar{z}_2| \leq \eps|z_1|^2 +\eps^{-1}|z_2|^2, \quad \forall\ \eps>0,
	$$
	we have 
	$$
		\beta\, |z_1|^2 + \gamma\, |z_1 + z_2|^2 \geq (\beta+ \gamma(1-\eps))|z_1|^2 + \gamma (1-\eps^{-1})|z_2|^2.
	$$
	The claim now follows upon setting $\eps= \frac{\beta+\gamma}{\gamma}$. 
\end{proof}

\begin{proof}[Proof of Theorem~\ref{thm:ev_diff}]
	Let $\phi_1$ be the positive normalized ground state of the Dirichlet Laplacian $H_{\Omega,0}^D$,
	\begin{align*}
		H_{\Omega,0}^D\phi_1 = \lambda_1(\Omega,0)\phi_1, \qquad \int_\Omega \phi_1^2\,dx = 1.
	\end{align*}
	To simplify the notation, we abbreviate $\lambda_1(\Omega,0) =\lambda_1$ and accordingly 
	for higher eigenvalues of $H_{\Omega,0}^D$. 
	
	For $u \in C^\infty_c(\Omega)$, we perform a groundstate substitution 
	\begin{align*}
		u(x) =: v(x)\phi_1(x), \quad v \in C^\infty_c(\Omega), 
	\end{align*} 
	so that 
	\begin{align*}
		h_{\Omega,A}^{D}[u] - \lambda_1\int_{\Omega}|u|^2\,dx = \int_{\Omega}|(-i\nabla + A)v|^2\,\phi_1^2\,dx
	\end{align*}
	by an explicit computation. Moreover, from the assumptions of the theorem it follows that 
	there exists $y\in \Omega$ and $\rho \in (0, \delta(y))$ such that $\Phi(\cdot, y)$ is not identically 
	zero in $(0,\rho)$. By Lemma~\ref{lem:ek} we know that for any 
	$R \in (0, \delta(y))$, there exists $F_1= F_1(y,B,R)\geq 0$ such that 
	\begin{equation}\label{eq:ek}
		\int_{\B(y,R)}\left|(-i\nabla + A)v\right|^2\,dx \geq F_1 \int_{\B(y,R)}|v|^2\,dx, \quad \forall \ v \in H^1(\Omega),
	\end{equation}
	and that  the function $F_1(y,B,\cdot)$ is not identically zero on $(0, \delta(y))$. We then write
	\begin{align}
		\int_{\Omega}|(-i\nabla + A)v|^2\, \phi_1^2\,dx 
		&\geq \frac{1}{2}\int_{\Omega}|(-i\nabla + A)v|^2\, \phi_1^2\, dx \label{1-source}\\
		&\quad + \frac{1}{2}\int_{\B(y,R)} |(-i\nabla + A)v|^2\,\phi_1^2\, dx. \nonumber
	\end{align}
	The last term is estimated as follows:
	\begin{multline*}
		\frac{1}{2}\int_{\B(y,R)}|(-i\nabla + A)v|^2\, \phi_1^2\, dx \geq \frac{1}{2}\left(\inf_{x \in \B(y,R)}
		\phi_1^2(x)\right) F_1 \int_{\B(y,R)}|v|^2\,dx\\
		\geq F_2 \int_{\B(y,R)}\phi_1^2\, |v|^2\,dx
		= F_2 \int_{\B(y,R)}|u|^2\,dx,
	\end{multline*}
	where
	\begin{align} \label{f2}
		F_2(y,B,R) := \frac 12\, F_1(y,B,R)\ \frac{\inf_{x \in \B(y,R)}\phi_1^2(x)}{\sup_{x \in \B(y,R)}\phi_1^2(x)}.
	\end{align}
	For the first term, we substitute back and use the diamagnetic inequality, so that
	\begin{align*}
		\frac{1}{2}\int_{\Omega} |(-i\nabla + A)v|^2\, \phi_1^2\, dx &= \frac{1}{2}\left(h_{\Omega,A}^D[u] - 
		\lambda_1 \int_{\Omega}|u|^2\,dx\right)\\
		&\geq \frac{1}{2}\int_{\Omega}\left|\nabla|u|\right|^2\,dx - \frac{\lambda_1}{2}\int_{\Omega}|u|^2\,dx.
	\end{align*}
	Putting the above estimates together, we obtain
	\begin{align*}
		h_{\Omega,A}^D[u] - \lambda_1\int_{\Omega}|u|^2\,dx  
		&\geq \frac{1}{2}\left(\int_{\Omega}\left|\nabla|u|\right|^2\,dx -\lambda_1\int_{\Omega}|u|^2\,dx\right) \\
		&\quad + F_2 \int_{\B(y,R)}|u|^2\,dx.
	\end{align*}
	Since we will be taking the infimum over all $u \in H_0^1(\Omega)$, we have by the inclusion of sets
	\begin{multline*}
		\inf_{\begin{subarray}{1}u \in H_0^1(\Omega)\\ \norm{u}_2 = 1 \end{subarray}}
		\left(h_{\Omega,A}^D[u] - \lambda_1\int_{\Omega}|u|^2\,dx\right)\\
		\geq \inf_{\begin{subarray}{1}u \in H_0^1(\Omega)\\ \norm{u}_2 = 1 \end{subarray}}
		\left(\frac{1}{2}\int_{\Omega}\left|\nabla|u|\right|^2\,dx 
		-\lambda_1\int_{\Omega}|u|^2\,dx + F_2 \int_{\B(y,R)}|u|^2\,dx\right)\\
		\geq \inf_{\begin{subarray}{1}w \in H_0^1(\Omega)\\ \norm{w}_2 = 1 \end{subarray}}
		\left(\frac{1}{2}\int_{\Omega}\left|\nabla w\right|^2\,dx 
		-\lambda_1\int_{\Omega}|w|^2\,dx + F_2 \int_{\B(y,R)}|w|^2\,dx\right).
	\end{multline*}
	Next, we observe that any $w \in H_0^1(\Omega)$ can be written as
	\begin{align*}
		w = \alpha\, \phi_1 + f, \quad \alpha = (w,\phi_1)_{L^2(\Omega)}
	\end{align*}
	so that $(\phi_1,f)=0$. Hence,
	\begin{align*}
		\frac{1}{2}\left(\int_{\Omega}\left|\nabla w\right|^2\,dx -\lambda_1\int_{\Omega}|w|^2\,dx\right) 
		&= \frac{1}{2}\left(\int_{\Omega}\left|\nabla f\right|^2\,dx -\lambda_1\int_{\Omega}|f|^2\,dx\right)\\
		&\geq \frac{\Delta \lambda}{2}\int_{\Omega}|f|^2\,dx,
	\end{align*}
	where $\Delta \lambda := \lambda_2 - \lambda_1 >0$. 
	From this we conclude that
	\begin{multline*}
		h_{\Omega,A}^D[u] - \lambda_1\int_{\Omega}|u|^2\,dx 
		\geq \frac{\Delta \lambda}{2}\int_{\Omega}|f|^2\,dx + F_2 \int_{\B(y,R)}|\alpha \phi_1 + f|^2\,dx\\
		\geq \frac{\Delta \lambda}{4} \int_{\Omega}|f|^2\,dx
		+\int_{\B(y,R)} \left(\frac{\Delta \lambda}{4}\, |f|^2 + F_2\,  |\alpha \phi_1 + f|^2\right) \,dx.
	\end{multline*}
	We then set $\beta = \frac{\Delta \lambda}{4}$ and apply to the last term Lemma~\ref{lem:c_nbr_est} 
	with $z_1=f, z_2=\alpha \phi_1$ and $\gamma=F_2$. This gives
	\begin{align*}
		h_{\Omega,A}^D[u] - \lambda_1\int_{\Omega}|u|^2\,dx 
		&\geq \frac{\Delta \lambda}{4}\int_{\Omega}|f|^2\,dx + |\alpha|^2\,  F_3 \int_{\B(y,R)}\phi_1^2\, dx, 
	\end{align*}
	with 
	\begin{equation}\label{f3}
		F_3= F_3(y,B,R)= \frac{F_2(y,B,R)\, \beta}{F_2(y,B,R) + \beta}.
	\end{equation}
	To sum up, we obtain the lower bound
	\begin{equation}\label{D-0}
		\lambda_1(\Omega,B) -\lambda_1(\Omega,0) 
		\geq\inf_{\substack{f \in H_0^1(\Omega) \\ \alpha \in \mathbb{C}}}\frac{\frac{\Delta \lambda}{4}\,\norm{f}_2^2 
		+|\alpha|^2\, F_3 \int_{\B(y,R)}\phi_1^2}{|\alpha|^2 + \norm{f}_2^2}.
	\end{equation} 
	The variational problem on the right hand side has an explicit solution:
	\begin{multline}\label{f3-lb}
		\inf_{\substack{f \in H_0^1(\Omega) \\ \alpha \in \mathbb{C}}}\frac{\frac{\Delta \lambda}{4}\, 
		\norm{f}_2^2 +  |\alpha|^2\, F_3 \int_{\B(y,R)}\phi_1^2}{|\alpha|^2 + \norm{f}_2^2}
		=  \inf_{t \geq 0}\frac{\frac{\Delta \lambda}{4}\,  + t\, F_3 \int_{\B(y,R)} \phi_1^2}{t+1}\\
		= \min\left\{\frac{\Delta \lambda}{4}\, , F_3  \int_{\B(y,R)}\phi_1^2 \right\}
		= F_3 \int_{\B(y,R)}\phi_1^2, 
	\end{multline}
	where the last equality follows from the definition of $F_3$ and the normalization of $\phi_1$. 
	Since \eqref{D-0} holds for any $R\leq \delta(y)$, it follows that lower bound \eqref{eq:ev_diff} holds with 
	\begin{align}\label{def:D_y}
		D (y,B) :=\sup_{0 < R < \delta(y)} F_3(y,B,R) \int_{\B(y;R)}\phi_1^2 . 
	\end{align}
\end{proof}

\begin{rem} 
	The value of $D(y,B)$ decreases when the support of the magnetic field gets closer to the boundary of $\Omega$, 
	which is natural. The precise value of $D(y,B)$ might in general depend in a complicated way on $B$ and on the 
	geometry of $\Omega$. 
\end{rem}

However, in the case of constant magnetic field it is possible to give a more explicit lower bound.  We will state the result separately for small and large values of the magnetic field.

\begin{corollary}\label{cor:ev_diff_1}
	Let $B=B_0>0$ be constant and assume that $ B_0 \leq R_{\inner}^{-2}$, then
	\begin{align} \label{lowerb-2-b0}
		\lambda_1(\Omega,B_0) - \lambda_1(\Omega,0) 
		&\geq \frac{B_0^2}{8}\, \sup_{y\in\Omega} \sup_{R\leq \delta(y)} \frac{Q(y,R)\,\Delta\lambda }{B_0^2\, Q(y,R) 
		+ 3\, \Delta\lambda}   \, \int_{\B(y;R)}\phi_1^2 ,
	\end{align}
	where 
	\begin{align*}
		Q(y,R) & := R^2\, \frac{\inf_{x \in \B(y,R)}\phi_1^2(x)}{\sup_{x \in \B(y,R)}\phi_1^2(x)}\, , \qquad 
		\Delta\lambda := \lambda_2(\Omega,0)-\lambda_1(\Omega,0)  .
	\end{align*}
\end{corollary}

\begin{proof}
	Let $y\in\Omega$. Assume first that  $B_0 \leq R_{\inner}^{-2}$ and let $R\leq \delta(y) \leq R_{\inner}$.
	A detailed inspection of Lemma~\ref{lem:ek} (see in particular equations \eqref{def:c0} and \eqref{def:c1})
	shows that in this case the quantity $F_1$ introduced in the proof of Theorem \ref{thm:ev_diff} 
	satisfies
	\begin{equation} \label{f1-lowerb-1}
		F_1 \geq \frac{B_0^2\, R^2}{12}, \qquad \text{if} \ \ B_0 \leq R_{\inner}^{-2}. 
	\end{equation}
	This in combination with \eqref{f2}, \eqref{f3} and \eqref{def:D_y} implies that 
	$$
		D(y,B_0) \geq \frac{B_0^2\, \Delta \lambda}{8}\frac{Q(y,R) \int_{\B(y;R)}\phi_1^2}{B_0^2\, Q(y,R) + 3\, \Delta\lambda}.
	$$
	Optimizing the right hand side first in $R$ and then in $y$ gives lower bound \eqref{lowerb-2-b0}.
\end{proof}

In order to state the result for larger values of $B_0$ we need some additional notation. 
Let $x_0\in\Omega$ be the center of a disc of radius $R_{\inner}$ contained in $\Omega$. 
It is easily seen that the disc $\B(x_0, R_{\inner}/2)$ contains 
\begin{equation} \label{NB}
	N(B_0) = [R_{\inner}^2\, B_0]
\end{equation}
disjoint squares of size $(2 B_0)^{-1/2}$. Let $y_j,\, j=1,\dots, N(B_0)$, be the centers of these squares. 
It follows that $\B(x_0, R_{\inner}/2)$, and therefore $\Omega$, contains $N(B_0)$ disjoint disc of radius 
\begin{align}\label{rho}
	\rho= \frac 12\, B_0^{-1/2}
\end{align}
centered in $y_j$. Let 
\begin{align*}
	Q_j :=\frac{\inf_{x \in \B(y_j, \rho)}\phi_1^2(x)}{\sup_{x \in \B(y_j, \rho)}\phi_1^2(x)}.
\end{align*}

\begin{corollary}\label{cor:ev_diff_2}
	Let $B=B_0>0$ be constant. Assume that $B_0 > R_{\inner}^{-2}$. Then 
	\begin{align} \label{lowerb-3-b0}
		\lambda_1(\Omega,B_0) - \lambda_1(\Omega,0) 
		&\geq \frac{B_0 }{8}\, \sum_{j=1}^{N(B_0)}  
		\frac{Q_j\, \Delta\lambda }{B_0\, Q_j + 12\, \Delta\lambda} \, \int_{\B(y_j, \rho)}\phi_1^2
	\end{align}
	where $N(B_0), y_j$ and $\rho$ are as above. 
\end{corollary}

\begin{proof}
	We will follow the proof of Theorem~\ref{thm:ev_diff} and replace the disc $\B(y,R)$ by the family of 
	discs $\B(y_j, \rho)$ with $j=1,2,\dots, N(B_0)$. Since the latter are disjoint by construction, we obtain a 
	modified version of inequality \eqref{1-source}:
	\begin{align}
		\int_{\Omega}|(-i\nabla + A)v|^2\, \phi_1^2\,dx 
		&\geq \frac{1}{2}\int_{\Omega}|(-i\nabla + A)v|^2\, \phi_1^2\, dx \label{n-sources}\\
		& \quad + \frac{1}{2}\, \sum_{j=1}^{N(B_0)} \int_{\B(y_j,\rho)} |(-i\nabla + A)v|^2\,\phi_1^2\, dx.\nonumber
	\end{align}
	Moreover, from Lemma~\ref{lem:ek} we deduce that for any $v \in H^1(\Omega)$ it holds that
	\begin{equation} \label{f1-lowerb-2}
		\int_{\B(y_j,\rho)} |(-i\nabla + A)v|^2\, dx \geq \frac{B_0}{48} \int_{\B(y_j,\rho)} |v|^2\,dx, \quad j=1,\dots, N(B_0). 
	\end{equation}
	Hence following the line of arguments of the proof of Theorem \ref{thm:ev_diff}, see equations \eqref{f2}--\eqref{f3-lb}, we arrive at 
	$$
		\lambda_1(\Omega,B) -\lambda_1(\Omega,0) \, \geq\,  \sum_{j=1}^{N(B_0)} F_{3,j} \int_{\B(y_j;\rho)}\phi_1^2, 
	$$
	where
	$$
		F_{3,j} \, \geq\, \frac{B_0 }{8}\,  \frac{Q_j\, \Delta\lambda }{B_0\, Q_j + 12\, \Delta\lambda}.
	$$
	The claim now follows.
\end{proof}
 
\begin{rem} 
	Note that by the Lebesgue property we have 
	$$
		\lim_{B_0\to \infty} B_0 \int_{\B(y_j;\rho)}\phi_1^2(x)\, dx  = \frac{\pi}{4}\, \phi_1^2(y_j).
	$$
	Hence in view of \eqref{NB} the right hand side of \eqref{lowerb-3-b0} remains bounded and strictly positive as $B_0$ tends to infinity. 
\end{rem}
\begin{rem}
	Note that when $\Omega$ is convex, the righthand side of Corollary~\ref{cor:ev_diff_1} and Corollary~\ref{cor:ev_diff_2} can 
	be further simplified by using the lower bound for the first spectral gap of the non-magnetic Laplacian.
	From \cite{ac} we know that
	\begin{align*}
		\Delta \lambda = \lambda_2(\Omega,0) - \lambda_1(\Omega,0) \geq \frac{3\pi^2}{\diam(\Omega)^2}.
	\end{align*}
\end{rem}

\subsection{Neumann boundary conditions}\label{ssec:second_type_est_n}
In the Neumann case, a simple perturbation argument with respect to the non-magnetic Laplacian shows 
that the corresponding estimate of \eqref{eq:comm_est_const} in the Neumann case must fail.
By taking the (normalized) constant function on $\Omega$ as a trial state we obtain
\begin{equation} \label{small-B-N}
	\mu_1(\Omega,B_0) = \mathcal{O}(B_0^2) \qquad B_0\to 0.
\end{equation}
This of course contradicts \eqref{eq:comm_est_const} with $\lambda_1(\Omega,B_0)$ replaced by $\mu_1(\Omega,B_0)$.

It is however possible to prove an analog of the estimate of the second type.
\begin{theorem} \label{thm-neumann}
	Let $\Omega\subset\R^2$ be a bounded open domain with Lipschitz boundary and let $B=B_0>0$. Then
	\begin{equation} \label{lowerb-N1}
		\mu_1(\Omega,B_0) 
		\geq \frac{\pi}{4\, |\Omega|}\ \frac{ B_0^2\, R_{\inner}^{4}\, \mu_2(\Omega,0)}{B_0^2\, R_{\inner}^{2}+ 6\, \mu_2(\Omega,0)} 
		\qquad \text{if} \ \ B_0 \leq R_{\inner}^{-2},
	\end{equation}
	and
	\begin{equation} \label{lowerb-N2}
		\mu_1(\Omega,B_0) 
		\geq \frac{\pi}{32\, |\Omega|}\ \frac{ N(B_0)\, \mu_2(\Omega,0)}{B_0 + 12\, \mu_2(\Omega,0)} 
		\qquad \text{if} \ \ B_0 > R_{\inner}^{-2},
	\end{equation}
	where $N(B_0)$ is given by \eqref{NB}.
\end{theorem} 

\begin{proof}
	Let $y\in\Omega$ be such that $\delta(y) = R_{\inner}$ and let $R\in (0,R_{\inner})$. 
	Let $\psi_1$ be the normalized eigenfunction of the Neumann Laplacian associated to eigenvalue $\mu_1(\Omega,0)=0$: 
	$$
		\psi_1(x) = |\Omega|^{-1/2} .
	$$
	We again use inequality \eqref{eq:ek} which, together with the diamagnetic inequality, leads to the lower bound
	\begin{equation} \label{1-lowerb-n}
		\mu_1(\Omega, B_0)
		\geq\frac12 \inf_{v\in H^1(\Omega)} \frac{ \int_\Omega |\nabla v|^2 +F_1 \int_{\B(y;R)} |v|^2}{\int_\Omega |v|^2}\, ,
	\end{equation}
	cp. \eqref{f2}. Next we use the decomposition
	\begin{equation}
	v = \alpha\, \psi_1 +f, \qquad \alpha= (v, \psi_1)_{L^2(\Omega)}. 
	\end{equation}
	Following the proof of Theorem~\ref{thm:ev_diff} and using the lower bound \eqref{f1-lowerb-1} we then find that 
	\begin{align*}
		\mu_1(\Omega, B_0) &\geq \frac{\mu_2\, F_1}{4 F_1 + 2 \mu_2}\int_{\B(y;R)} \psi_1^2\,dx
		= \frac{\pi}{ |\Omega|}\, \frac{\mu_2 R^2\, F_1}{4 F_1 + 2 \mu_2}\\
		&\geq\frac{\pi}{ 4\, |\Omega|}\frac{B_0^2\, R^4 \mu_2 }{B_0^2 R^2 + 6\, \mu_2},
	\end{align*}
	where we used the abbreviation $\mu_2 := \mu_2(\Omega,0)$. Optimizing in $R$ then gives \eqref{lowerb-N1}. 
	As for inequality \eqref{lowerb-N2}, this follows by mimicking the proof of Corollary~\ref{cor:ev_diff_2}. 
	Indeed, if we replace $\phi_1$ by $\psi_1$ and $\Delta\lambda$ by  $\mu_2(\Omega,0)$ we  end up with \eqref{lowerb-N2}. 
\end{proof}

\begin{rem}
	Notice that both lower bounds \eqref{lowerb-2-b0} and \eqref{lowerb-N1} are for 
	small values of $B_0$ proportional to $B_0^2$. 
	This is in agreement with the asymptotic expansions \eqref{eq:small_B_asymp} and \eqref{small-B-N}.
\end{rem}

\appendix
\section{Appendix}\label{sec:app} 
Let $y\in\Omega$ and let $R\in (0,\delta(y))$. Let 
\begin{equation*}
	\mu(r) = \min_{k\in\Z} |k-\Phi(r,y)|,
\end{equation*}
where $\Phi(r,y)$ is given by \eqref{def:flux_local}. Define the function $\chi: [0,R] \to [0,1]$ by 
\begin{equation*}
	\chi(r) := \frac{\mu_0^2\, \mu^2(r)}{r^2}\, , \qquad \mu_0 := \left(\max_{[0,R]} \, \frac{\mu(r)}{r} \right)^{-1},
\end{equation*}
and let 
$$
	\nu_0 := \max_{[0,R]} \, |r^{-2}\, (r\mu'(r)-\mu(r))|. 
$$
From the definition of $\chi$ it follows that there exists at least one $r_0\in [0,R]$ such that 
\begin{equation}\label{def:r_0}
	\chi(r_0)=1. 
\end{equation}
To any such $r_0$ we define the constants
\begin{align} 
	c_0  &:= 4\, \max \left\{ j^{-2}_{0,1}\, r_0^2, (6 r_0)^{-1}(2 R^3-3R^2 r_0 +r_0^3)\right\},\label{def:c0}\\
	c_1  &:= \max \left\{ 2\mu_0^2+ 4 c_0 \nu_0^2\, \mu_0^4, c_0\right\}.\label{def:c1}
\end{align}
With this notation we can state \cite[Lemma~3.1]{ek}:

\begin{lemma}[Ekholm-Kova\v r\'{\i}k]\label{lem:ek}
	Let $B\in L^\infty_{loc}(\R^2)$ and let $r_0$ satisfy \eqref{def:r_0}. Then any $v\in H^1(\Omega)$ satisfies
	\begin{equation}\label{eq:ek_lem}
		c_1 \int_{\B(y,R)}\left|(-i\nabla + A)v\right|^2\,dx \geq \int_{\B(y,R)}|v|^2\,dx.
	\end{equation}
\end{lemma}

\section*{Aknowledgements}
The authors would like to thank S.~Fournais, B.~Helffer, N.~Raymond and J.~P.~Solovej for valuable discussions.
Part of this work has been carried out during a visit at the Mathematisches Institut Oberwolfach during the Program
``Eigenvalue Problems in Surface Superconductivity".
T.~E. was supported by the Swedish Research Council grant Nr. 2009-6073. H.~K. was supported by the Gruppo Nazionale per Analisi Matematica, la Probabilit\`a e le loro Applicazioni (GNAMPA) of the Istituto Nazionale di Alta Matematica (INdAM).
The support of MIUR-PRIN2010-11 grant for the project  ``Calcolo delle variazioni'' (H.~K.) is also gratefully acknowledged. 
F.~P. acknowledges support from the Swedish Research Council grant Nr. 2012-3864 and 
ERC grant Nr. 321029 ``The mathematics of the structure of matter".


\end{document}